\documentclass[a4paper,11pt,parskip=half]{amsart}

 \usepackage[foot]{amsaddr}
\usepackage[T1]{fontenc}
\usepackage[utf8]{inputenc}
\usepackage{palatino}
\usepackage{amsmath, amssymb,mathtools}
\mathtoolsset{showonlyrefs}
\usepackage[usenames,dvipsnames]{xcolor}
\usepackage{graphicx}
\usepackage{subfigure}
\usepackage{minitoc}
\usepackage{tikz}
\usepackage{enumerate}
\usepackage{filecontents}
\usepackage[margin=0.96 in]{geometry}
\usepackage{bbm}
\usepackage[numbers,sort&compress]{natbib}
\usepackage[colorlinks=true]{hyperref}
\hypersetup{urlcolor=blue, citecolor=red}

\DeclarePairedDelimiter{\abs}{\lvert}{\rvert}
\DeclarePairedDelimiter{\norm}{\lVert}{\rVert}
\DeclarePairedDelimiter{\set}{\{}{\}}

\DeclareMathAlphabet{\mathup}{OT1}{\familydefault}{m}{n}
\newcommand{\dx}[1]{\mathop{}\!\mathup{d} #1}

\DeclarePairedDelimiter{\prt}{(}{)}
\DeclarePairedDelimiter{\brk}{[}{]}

\newcommand{\N}{{\mathbb N}}
\newcommand{\R}{{\mathbb R}}

\usepackage{amsthm}
\theoremstyle{plain}
\newtheorem{theorem}{Theorem}[section]
\newtheorem{lemma}[theorem]{Lemma}

\newtheorem{corollary}[theorem]{Corollary}

\theoremstyle{remark}
\newtheorem{remark}[theorem]{\bf Remark}

\newcommand{\ie}{\emph{i.e.}}
\newcommand{\cf}{\emph{cf.}\;}
\renewcommand{\i}{^{(i)}}
\newcommand{\1}{^{(1)}}
\newcommand{\2}{^{(2)}}

\newcommand{\lip}{\mathrm{Lip}}
\newcommand{\sign}{\mathrm{sign}}
\newcommand{\ds}{\displaystyle}
\newcommand{\ddt}{\frac{\dx{}}{\dx{t}}}
\newcommand{\partialt}[1]{\frac{\partial #1}{\partial t}}
\newcommand{\partialx}[1]{\frac{\partial #1}{\partial x}}
\newcommand{\fpartial}[1]{\frac{\partial}{\partial #1}}
\newcommand{\grad}{\nabla}
\renewcommand{\div}{\nabla\cdot}
\newcommand{\Lap}{\Delta}

\begin{document}

\title[Incompressible limit for a viscous two-species tumour model  in 1D]{Incompressible limit for a two-species tumour  model with coupling through Brinkman's law in one dimension}

\author{Tomasz D\k{e}biec$^{1}$}
\author{Markus Schmidtchen$^{2}$}

\address{$^{1}$ Institute of Applied Mathematics and Mechanics, University of Warsaw, Banacha 2, 02-097 Warsaw, Poland (t.debiec@mimuw.edu.pl).}
\address{$^{2}$ Sorbonne Universit\'es, UPMC University of Paris, Paris 75005, France (markus.schmidtchen@upmc.fr).}

\maketitle
\begin{abstract}
    We present a two-species model with applications in tumour modelling. The main novelty is the coupling of both species through the so-called Brinkman law which is typically used in the context of visco-elastic media, where the velocity field is linked to the total population pressure via an elliptic equation. The same model for only one species has been studied by Perthame and Vauchelet in the past. The first part of this paper is dedicated to establishing existence of solutions to the problem, while the second part deals with the incompressible limit as the stiffness of the pressure law tends to infinity. Here we present a novel approach in one spatial dimension that differs from the kinetic reformulation used in the aforementioned study and, instead, relies on uniform BV-estimates.
\end{abstract}{}

%%%%%%%%%%%%%%%%%%%%%%%%%%%%%%%%%%%%%%%%%%%%%%%%%%%%%%%%%%%%%%%%%%%%%%%%%%%%%%%%%%%%%%%%%%
%%%%%%%%%%%%%%%%%%%%%%%%%%%%%%%%%%%%%%%%%%%%%%%%%%%%%%%%%%%%%%%%%%%%%%%%%%%%%%%%%%%%%%%%%%
%%%%%%%%%%%%%%%%%%%%%%%%%%%%%%%%%%%%%%%%%%%%%%%%%%%%%%%%%%%%%%%%%%%%%%%%%%%%%%%%%%%%%%%%%%
\section{Introduction}
In recent years there has been an increasing interest in multi-phase models applied to tumour growth. Traditionally, tumour growth was modelled using a single equation describing the evolution of the abnormal cell density. This paper is dedicated to studying the two-species model 
\begin{align*}
    \left\{
    \begin{array}{rl}
	\ds \partialt{n\i_k} - \div\prt*{n\i_k \grad W_k} &= \ds n\i_k G\i(p_k), \\[1em]
    \ds -\nu \Lap W_k +W_k &= \ds  p_k,
	\end{array}
	\right.
\end{align*}
where $n\i$ represents the normal (resp.\ abnormal) cells, for $i=1,2$, and $k\in \N$ is a given constant modelling the stiffness of the total population pressure, $p_k$, which is generated by both species, \ie,
\begin{align*}
	p_k := \frac{k}{k-1}\prt*{n_k\1 + n\2_k}^{k-1}.
\end{align*}
In addition, $\nu>0$ is a fixed positive constant that is understood as a measure of viscosity. The elliptic equation linking the macroscopic velocity, $W_k$, with the pressure $p_k$ is typically referred to as \emph{Brinkman's law}, for instance \cf \cite{All91}.
The growth of the two densities is assumed to be modulated by two functions $G\i$, for $i=1,2$, that are assumed to be decreasing in their variable, $p_k$, similar to \cite{BD09, RBEJPJ10}.

Throughout, we shall use the shorthand notation $n_k := n\1_k + n_k\2$, in order to denote the total population. Upon adding up the two equations for the individual species, we obtain an equation for the total population density, $n_k$, \ie,
\begin{align}
    \label{eq:totalpopulation}
	\partialt{n_k} - \div\prt{n_k \grad W_k} = n_k \prt*{r_k G\1(p_k) + (1-r_k) G\2(p_k)},
\end{align}
where $r_k$ is the population fraction $r_k:= n_k\1 / n_k$. Related models have been extensively studied in the past. We refer to \cite{PQTV14, PV15}, and references therein, for a treatise of the incompressible limit for a single-species visco-elastic tumour model. As above, the velocity field is given by an elliptic equation involving the pressure that, in their case, is just given by a power of the sole species. Introducing the coupling of the two equations for the individual species drastically changes the behaviour and the same tool employed in \cite{PV15} cannot be applied, at least not in a straightforward manner, and a different strategy has to be found. Even in the case $\nu = 0$ corresponding to the inviscid case, the system nature of the problem gives rise to a whole range of difficulties, \cf~\cite{CFSS17, GPS19, BPPS19}. At first glance, the pressure gains in regularity, however, it gains just enough regularity to obtain compactness of its gradient, requiring a minute derivation of suitable estimates. Let us stress that the same type of difficulties are also encountered when the pressure is not given as a power law, \cf~\cite{HV17, DHV18, CDHV18}. A key tool in obtaining existence results and stable (with respect to the parameter $k$) estimates is to devise and manipulate the equation satisfied by the (joint) population pressure, \cf~\cite{CFSS17, GPS19, BPPS19, PV15, MPQ17, PQV14, HV17, DHV18, CDHV18}. In this work we shall follow this path. An easy application of the chain rule in conjunction with Eq.~\eqref{eq:totalpopulation} leads to
\begin{align*}
	\partialt{p_k} - \grad p_k \cdot \grad W_k =\frac{k-1}{\nu}p_k \brk*{W_k - p_k + \nu r_k G\1(p_k) + \nu (1-r_k)G\2(p_k)},
\end{align*}
where the population fraction $r_k$ satisfies
\begin{align*}
	\partialt {r_k} -\grad r_k \cdot \grad W_k = r_k (1-r_k)\brk*{G\1(p_k) - G\2(p_k)}.
\end{align*}
The change to these new variables was first introduced in \cite{BGH87, BGHP85, BGH87a} in the context of a two-species system where the two species avoid overcrowding. In a way, their works paved the way for more modern approaches to tumour models linked through Darcy's law, \cf \cite{BHIM12, GPS19, CFSS17, BPPS19}.

\medskip

The rest of this paper is organised as follows. In the subsequent section we set up precisely the problem and state our assumptions. In Section \ref{sec:Existence} we establish existence of solutions to the main system under consideration, Eq.~\eqref{eq:system}, and discuss their regularity necessary for our purposes. Section \ref{sec:Apriori} is dedicated to establishing a range of a priori estimates necessary in the analysis of the incompressible limit. Section \ref{sec:StrongCompactnessOfPressure} is devoted to establishing the strong compactness of the pressure, which is  key in passing to the stiff limit. Finally, with all information at hand, we pass to the incompressible limit in the pressure equation and derive the so-called \emph{complementarity relation} in Section \ref{sec:Incompressible}. We round off the analytical results in Section \ref{sec:Numerics} by presenting some numerical simulations for different parameter choices.

%%%%%%%%%%%%%%%%%%%%%%%%%%%%%%%%%%%%%%%%%%%%%%%%%%%%%%%%%%%%%%%%%%%%%%%%%%%%%%%%%%%%%%%%%%
%%%%%%%%%%%%%%%%%%%%%%%%%%%%%%%%%%%%%%%%%%%%%%%%%%%%%%%%%%%%%%%%%%%%%%%%%%%%%%%%%%%%%%%%%%
%%%%%%%%%%%%%%%%%%%%%%%%%%%%%%%%%%%%%%%%%%%%%%%%%%%%%%%%%%%%%%%%%%%%%%%%%%%%%%%%%%%%%%%%%%
\section{Preliminaries and Statement of the Main Results}
We study the system
\begin{subequations}
\label{eq:system}
\begin{align}
    \left\{
    \begin{array}{rl}
	\ds \partialt{n\1_k} - \fpartial x \prt*{n\1_k \partialx {W_k}} &= \ds n\1_k G\1(p_k), \\[1em]
	\ds \partialt{n\2_k} - \fpartial x \prt*{n\2_k \partialx {W_k}} &= \ds n\2_k G\2(p_k),
	\end{array}
	\right.
\end{align}
posed on the whole domain $\R$. It is coupled through the Brinkman law
\begin{align}
    \label{eq:brinkman}
    \ds -\nu \frac{\partial^2}{\partial x^2} W_k +W_k &= \ds  p_k.
\end{align}
\end{subequations}
The system is equipped with  non-negative initial data 
\begin{subequations}
\label{eq:InitialData}
\begin{align}
    n\i_{0,k} \in L^1(\R) \cap L^\infty(\R),
\end{align}
for any integer $k\geq2$.  Moreover, we assume that there exists a constant, $C>0$, such that
\begin{align}
    \int_\R\; \abs*{\partialx {n\i_{0,k}}} \dx{x} \leq C,
\end{align}
for $i=1,2$, and every $k\geq2$.
\end{subequations}
As before, the pressure is given in form of a power of the joint population, \ie,
\begin{align}
    p_k := \frac{k}{k-1} \prt*{n_k\1 + n_k\2}^{k-1} =\frac{k}{k-1} n_k^{k-1}.
\end{align}
Recall that the pressure satisfies 
\begin{align}
    \label{eq:PressureEqn}
    \partialt{p_k} - \partialx {p_k}  \partialx {W_k} =\frac{k-1}{\nu}p_k \brk*{W_k - p_k + \nu r_k G\1(p_k) + \nu (1-r_k)G\2(p_k)},
\end{align}
with the population fraction, $r_k := n\1_k / n_k$, given by 
\begin{align}
    \label{eq:RatioEqn}
    \partialt {r_k} -\partialx {r_k}  \partialx{W_k} = r_k (1-r_k)\brk*{G\1(p_k) - G\2(p_k)}.    
\end{align}{}

Throughout the paper we assume the following regularity and properties of the growth functions $G\i$, $i=1,2$,
\begin{align}
    \label{eq:AssumptionsOnG}
    G\i \in C^1(\R),\quad G\i_p \leq -\alpha < 0,\quad G\i(p_M) = 0,
\end{align}
for some  $p_M > 0$, where $G\i_p$ denotes the derivative of the function $G\i$. The pressure $p_M$ is often called the \emph{homeostatic pressure}.

Recall that a solution $W_k$ to  Brinkman's equation $-\nu\partial^2_x W_k + W_k = p_k$ can be written as $W_k = K\star p_k$, where $K$ is the fundamental solution to the equation $-\nu{\partial_x^2 K} + K = \delta_0$, \ie,
\begin{equation}
    \label{eq:SourceSolution}
    K(x) = \frac{1}{4\pi}\int_0^\infty\exp{[-(\pi|x|^2\slash {4s\nu} + s\slash{4\pi})]}s^{-1/2}\dx{s} = \frac{1}{2\sqrt{\nu}} \exp\prt*{-\nu^{-1/2}|x|}.
\end{equation}
Then $K\geq 0$,\; $\int K(x)\dx{x} = 1$ and $K, \partial_x{K} \in L^q(\R)$ for $1\leq q \leq \infty$.
By the elliptic regularity theory we have $W_k(t,\cdot)\in W^{2,q}(\R)$, for any $t\in[0,T]$, $1\leq q\leq\infty$.

\bigskip

Below we formulate the main results of this work.

\begin{theorem}[Existence of Solutions]
    \label{thm:Existence}
    For any initial data satisfying \eqref{eq:InitialData}, system \eqref{eq:system} admits a solution $n_k\1, n_k\2 \in L^\infty(0,T; BV(\R) \cap L^\infty(\R))$.
\end{theorem}
We highlight the fact that solutions are essentially bounded since these bounds are not a consequence of the BV-bounds. Rather, they are obtained independently. This may prove useful for an extension to higher dimensions in future works.

\begin{theorem}[Incompressible Limit and Complementarity Relation]
    \label{thm:IncompressibleLimit}
    We may pass to the limit $k\to \infty$ in the pressure equation, Eq. \eqref{eq:PressureEqn}. This yields the so-called complementarity relation
    \begin{align}
        0 = p_\infty \prt*{W_\infty - p_\infty + \nu n\1_\infty G\1(p_\infty) + \nu n\2_\infty G\2(p_\infty)},
    \end{align}
    in the distributional sense, where $n\i_\infty$, $i=1,2$, satisfies
\begin{align*}
    \left\{
    \begin{array}{rl}
    \ds \partialt{n\i_\infty} - \partialx{}{\prt*{n\i_\infty\partialx{W_\infty}}} &=\ds n\i_\infty G\i(p_\infty),\\[1em]
    \ds -\nu \frac{\partial^2 W_\infty}{\partial x^2} + W_\infty &=\ds p_\infty.
    \end{array}
    \right.
\end{align*}
    Moreover, the following holds true
    \begin{align*}
        p_\infty(n_\infty - 1) = 0.
    \end{align*}
\end{theorem}
The subsequent sections are concerned with the proof of the two main theorems.

%%%%%%%%%%%%%%%%%%%%%%%%%%%%%%%%%%%%%%%%%%%%%%%%%%%%%%%%%%%%%%%%%%%%%%%%%%%%%%%%%%%%%%%%%%
%%%%%%%%%%%%%%%%%%%%%%%%%%%%%%%%%%%%%%%%%%%%%%%%%%%%%%%%%%%%%%%%%%%%%%%%%%%%%%%%%%%%%%%%%%
%%%%%%%%%%%%%%%%%%%%%%%%%%%%%%%%%%%%%%%%%%%%%%%%%%%%%%%%%%%%%%%%%%%%%%%%%%%%%%%%%%%%%%%%%%
\section{Existence of Solutions and Regularity}
\label{sec:Existence}
This section is dedicated to proving the existence of solutions to the $(p,r)$-system. The proof is based on an application of Banach's fixed point theorem. Let $k\geq2$ be fixed throughout this section. Further, assume for now, that the initial data $u_0\i$ are Lipschitz continuous. For given functions $p,r \in L^\infty(0,T; L^\infty(\R))$ we construct solutions $u\i$ to the linearised system, $i=1,2$,
\begin{align}
    \label{eq:Existence_LinearisedSystem}
	\ds \partialt u\i - \partialx { u\i}  \partialx W =\ds \frac{k-1}{\nu} f\i (p,r),
\end{align}
where 
\begin{align*}
	f^1(p,r) = K\star p - p + \nu r G\1(p) + \nu (1-r)G\2(p),
\end{align*}
and
\begin{align*}
	f^2(p,r) = r(1-r) \brk{G\1(p) - G\2(p)}.
\end{align*}

For the fixed $p$ from above, we may construct the backward flow
\begin{align*}
	\left\{
	\begin{array}{rl}
		\ds\frac{\dx{}{X_{(x,t)}}}{\dx{s}}(s) &= -\ds  \partialx {W} (X_{(x,t)},s),\\[1em]
		X_{(x,t)}(t) &= x.
	\end{array}
	\right.
\end{align*}
We readily observe that 
\begin{align*}
	u\i(t,x) = u_0\i(X_{(x,t)}(s=0)) + \int_0^t f\i(p(\tau,x), r(\tau,x))\dx{\tau},
\end{align*}
$i=1,2$, solve the linearised system~\eqref{eq:Existence_LinearisedSystem}. Now, considering another element $(\tilde p,\tilde r)$ in $L^\infty(0,T; L^\infty(\R))$, we observe that
\begin{align*}
	\abs*{u\i(t,x) - \tilde u\i(t,x)} 
	&= \abs*{u_0\i(X_{(x,t)}(s=0)) - u_0\i(\tilde X_{(x,t)}(s=0))}\\
	&\leq \lip(u_0\i)\abs*{X_{(x,t)}(s=0) - \tilde X_{(x,t)}(s=0)}\\
	&\leq \lip(u_0\i)\int_t^0\abs*{\partialx  W(X_{(x,t)}(s),s)  - \partialx {\tilde W}(\tilde X_{(x,t)}(s),s)}\dx{s}\\
	&\leq \lip(u_0\i)\int_t^0\abs*{\partialx  K \star(p-\tilde p)}\dx{s}\\
	&\leq \lip(u_0\i)\int_t^0\norm*{\partialx K}_{L^1} \norm{p-\tilde p}_{L^\infty} \dx{s}\\
	&\leq \lip(u_0\i) T \norm*{\partialx  K}_{L^1} \norm{p-\tilde p}_{L^\infty}.
\end{align*}
Thus, upon passing to the supremum, we obtain the following stability estimate for two solutions
\begin{align}
   \norm*{u\i - \tilde u\i}_{L^\infty} \leq C T \norm{p-\tilde p}_{L^\infty}.
\end{align}
In particular, for $T_1>0$ small enough the estimate gives rise to a contraction in the Banach space $L^\infty(0,T_1; L^\infty(\R))$, which is sufficient to infer the existence of a unique fixed point, by an application of Banach's fixed point theorem. Since the supremum norm of the solution does not blow up, a finite number of iterations of the above argument leads to existence of solutions for all times $T>0$.\\

For the subsequent analysis, let us call this fixed point $(u\1_\ast, u\2_\ast)$. It remains to prove the expected BV-regularity of solutions. This is an easy consequence of the ``transport nature'' of the system, \ie, 
\begin{align}
    \fpartial t \partialx {u\i_\ast} &= \partialx {u\i_\ast} \partialx {W} + \frac{k-1}{\nu} \brk*{ f\i_p(u\1_\ast, u\2_\ast) \partialx {u\1_\ast} +f\i_r(u\1_\ast, u\2_\ast) \partialx {u\2_\ast} }.
\end{align}
Multiplying by $\sign(u\i_\ast)$ and adding the two equations, for $i=1,2$, we obtain, after integrating
\begin{align}
    \ddt \int_\R \abs*{\partialx {u\1_\ast} } + \abs*{\partialx {u\2_\ast}} \dx{x} \leq C \int_\R \abs*{\partialx {u\1_\ast} } + \abs*{\partialx {u\2_\ast}} \dx{x},
\end{align}
where the constant $C>0$ depends only on the Lipschitz constants of the functions $f\i$ and the $L^\infty$-bounds on the fixed point. In particular, from Gronwall's inequality we deduce a control on the BV-seminorm and, more importantly, the existence of solutions even in cases where $u\i_0$ is not Lipschitz continuous but only of bounded variation.

Using the fact that
\begin{equation*}
    n\1 = \prt*{\frac{k-1}{k}u\1_*}^{\frac1{k-1}}u_*\2,\quad \text{and } \quad n\2 = \prt*{\frac{k-1}{k}u\1_*}^{\frac1{k-1}}\prt*{1-u_*\2},
\end{equation*}
the existence result transfers to the original system for $n\i_k$, $i=1,2$.

\begin{remark}[Extension to Higher Dimensions]
Let us remark here that the same strategy can be easily extended to higher dimensions since the transport nature is the same in any dimension. In fact, the only ``problematic'' point in our strategy is the contraction argument which depends on $\norm{\partial_x K}_{L^1}$. However, this norm is finite in any dimension, and therefore our existence result holds in any dimension.
\end{remark}

%%%%%%%%%%%%%%%%%%%%%%%%%%%%%%%%%%%%%%%%%%%%%%%%%%%%%%%%%%%%%%%%%%%%%%%%%%%%%%%%%%%%%%%%%%
%%%%%%%%%%%%%%%%%%%%%%%%%%%%%%%%%%%%%%%%%%%%%%%%%%%%%%%%%%%%%%%%%%%%%%%%%%%%%%%%%%%%%%%%%%
%%%%%%%%%%%%%%%%%%%%%%%%%%%%%%%%%%%%%%%%%%%%%%%%%%%%%%%%%%%%%%%%%%%%%%%%%%%%%%%%%%%%%%%%%%
\section{A Priori Estimates}
\label{sec:Apriori}
In this section we derive some bounds for the main quantities of interests, uniformly in $k$. These will be vital when passing to the limit with $k\to\infty$.

\begin{lemma}[A priori estimates I]
    \label{lemma:apriori}
	The following hold uniformly in $k$ for any $T>0$.
	\begin{enumerate}[(i)]
		\item $n_k \in L^{\infty} (0,T; L^1(\R))$,
		\item $p_k \in L^{\infty} (0,T; L^{\infty}(\R))$,
		\item $n_k \in L^{\infty} (0,T; L^{\infty}(\R))$,
		\item $p_k \in L^{\infty} (0,T; L^{1}(\R))$, and 
		\item $n_k\i \in L^{\infty} (0,T; L^{\infty}(\R))$, for $i=1,2$.
	\end{enumerate}
\end{lemma}
\begin{proof}
Clearly when $n_k(t=0)\geq0$, then $n_k$ stays non-negative at all times. Integrating Eq.~\eqref{eq:totalpopulation} in space and time we deduce that $n_k\in L^{\infty} (0,T; L^1(\R))$ uniformly in $k$. By the maximum principle we have the bound $0\leq p_k \leq p_M$. Then using $n_k \simeq p^{\frac1{k-1}}$ we deduce $n_k \in L^{\infty} (0,T; L^{\infty}(\R))$ uniformly. Writing $p_k \leq n_k\norm{n_k}_\infty^{k-2}$ we see that $p_k\in L^{\infty} (0,T; L^{1}(\R))$. Finally, we use that $n\1_k = r_k n_k$ and $0\leq r_k\leq1$ to deduce the last bounds.
\end{proof}

Using the above Lemma and the boundedness of $W_k$, we have the following result.
\begin{lemma}[Integrability and Segregation]
    \label{lem:IntegrabilityAndSegregation}
	If both species are segregated initially, \ie, 
	\begin{align*}
		\int_\R r_k^0(1-r_k^0) \dx{x} = 0,
	\end{align*}
	then there holds
	\begin{align*}
		\int_\R r_k(t,x)\,\prt*{1-r_k(t,x)} \dx{x} = 0,
	\end{align*}
	for all times $0\leq t \leq T$. 
    In particular, $r_k^0(1-r_k^0) \in L^1(\R)$ implies $r_k(1-r_k) \in L^\infty(0,T; L^1(\R))$.
\end{lemma}
\begin{proof}
Here and henceforth we shall employ the notation
\begin{equation*}
    \norm*{G\i}_{\infty}\coloneqq \sup_{0\leq p \leq P_M}\abs*{G\i(p)}.
\end{equation*}
The supremum is taken only up to $p_M$, because in principle the functions $G\i$ can decrease arbitrarily. The uniform bound obtained in the previous proof shows however, that only the range $0\leq p_k\leq p_M$ is relevant.

Using the equation for the population fraction and boundedness of the growth functions $G\i$, we obtain
\begin{align*}
	\ddt \int_\R r_k \,\prt{1-r_k}\dx{x} 
	&= \int_\R(1 - 2r_k) \prt*{\partialx{r_k}\partialx{W_k} + r_k (1 - r_k) \brk*{G\1(p_k) - G\2(p_k)}}\dx{x}\\
	&\leq \max_{i=1,2}\,\norm*{G\i}_{\infty}\, \int_\R r_k (1-r_k)\dx{x} + \int_\R \partialx{}\prt*{r_k (1-r_k)} \partialx{W_k} \dx{x}\\
	&\leq \max_{i=1,2}\,\norm*{G\i}_{\infty}\, \int_\R r_k (1-r_k)\dx{x} - \int_\R r_k (1-r_k) \frac{\partial^2 W_k}{\partial x^2} \dx{x}.
\end{align*}
Using Brinkman's law~\eqref{eq:brinkman}, we obtain
\begin{align*}
	\ddt \int_\R r_k \,\prt{1-r_k}\dx{x} 
	&\leq \max_{i=1,2}\,\norm*{G\i}_{\infty}\, \int_\R r_k (1-r_k)\dx{x} + \int_\R r_k (1-r_k) \frac{p_k-W_k}{\nu} \dx{x}\\
	&\leq C \int_\R r_k (1-r_k)\dx{x},
\end{align*}
having used the a priori bounds on the pressure, $p_k$.
\end{proof}

The following lemma establishes an $L^1$-bound on the right-hand side of the pressure equation.

\begin{lemma}[A priori estimates II]
\label{lemma:pQinL1}
    The following estimate holds for any $T>0$
    \begin{equation*}
    k\int_0^T\int_{\R} p_k |W_k - p_k + \nu r_k G\1(p_k) + \nu (1-r_k)G\2(p_k)| \dx{x}\dx{t} \leq C(T),
\end{equation*}
for a constant $C(T)>0$, independent of $k$.
Furthermore, the following bounds hold uniformly in $k$
    \begin{enumerate}[(i)]
%        \item $W_k \in L^\infty(0,T;W^{1,q}(\R))$ for $1\leq q \leq\infty$,
        \item $\ds \partialt{W_k} \in L^1(0,T;L^q(\R))$,\; for $1\leq q\leq\infty$,\\
        \item $\ds \partialt{}{\partialx{W_k}} \in L^1(0,T;L^q(\R))$,\; for $1< q <\infty$.
    \end{enumerate}
\end{lemma}

\begin{proof}
Let us introduce the following notation 
\begin{equation*}
    Q_k\coloneqq W_k - p_k + \nu r_k G\1(p_k) + \nu (1-r_k)G\2(p_k),
\end{equation*}
and follow the strategy of~\cite{PV15}.
Using
\begin{equation}\label{eq:tshiftofW}
    \partialt{W_k} = K \star \brk*{\partialx{p_k}\partialx{W_k} + \frac{k-1}{\nu}p_k Q_k},
\end{equation}
we derive the equation
\begin{equation*}
    \begin{aligned}
    \partialt{Q_k} &- \partialx{Q_k} \partialx{W_k} + \frac{k-1}{\nu}p_kQ_k\brk*{1-r_k G\1_p(p_k)-\prt*{1-r_k} G\2_p(p_k)}\\
    &= -\abs*{\partialx{W_k}}^2 + K \star \brk*{\partialx{p_k}\partialx{W_k} + \frac{k-1}{\nu}p_k Q_k} + \nu\prt*{G\1(p_k) - G\2(p_k)}^2r_k(1-r_k),
    \end{aligned}
\end{equation*}
and consequently,
\begin{equation*}
    \begin{aligned}
    \partialt{\abs*{Q_k}} &- \partialx{\abs*{Q_k}} \partialx{W_k} + \frac{k-1}{\nu}p_k\abs*{Q_k}\brk*{1-r_k G\1_p(p_k)-(1-r_k) G\2_p(p_k)}\\ 
    &\leq -\abs*{\partialx{W_k}}^2 + \abs*{K \star \brk*{\partialx{p_k}\partialx{W_k} + \frac{k-1}{\nu}p_k \abs*{Q_k}}} + \nu\prt*{G\1(p_k) - G\2(p_k)}^2r_k\prt*{1-r_k}.
    \end{aligned}
\end{equation*}
Integrating in space-time and using the assumption that $|G\i_p|\geq \alpha > 0$, we obtain
\begin{equation*}
    \begin{aligned}
    \alpha(k-1)\int_0^T\int_\R p_k\abs*{Q_k} \dx{x}\dx{t} & \leq \int_\R\abs*{Q_k(x,0)} - \abs*{Q_k(x,T)} \dx{x} + \int_0^T\int_\R\abs*{\partialx{W_k}}^2 \dx{x}\dx{t}\\
    &\;\; + \nu^{-1}\int_0^T\int_\R\abs*{Q_k}\prt*{p_k - W_k} + \abs*{K \star \brk*{\partialx{p_k}\partialx{W_k}}}  \dx{x}\dx{t}\\
    &\;\; + \nu  \int_0^T\int_\R \nu\prt*{G\1(p_k) - G\2(p_k)}^2r_k(1-r_k) \dx{x}\dx{t}.
    \end{aligned}
\end{equation*}
The first three terms on the right-hand side are controlled uniformly, as is the very last term. For the remaining two terms we write
\begin{equation*}
    \nu^{-1}\int_0^T\int_\R|Q_k|(p_k - W_k)\dx{x}\dx{t} \leq \nu^{-1}\int_0^T\int_\R|Q_k|p_k\dx{x}\dx{t},
\end{equation*}
which, for $k$ large enough, is controlled by the left-hand side of the last inequality, and
\begin{equation*}
    K \star \brk*{\partialx{p_k}\partialx{W_k}} = \partialx{K} \star \brk*{p_k\partialx{W_k}} - K \star \brk*{p_k\frac{\partial^2 W_k}{\partial x^2}} =  \partialx{K} \star \brk*{p_k \partialx{K}\star p_k} - K \star \brk*{p_k\frac{\partial^2 W_k}{\partial x^2}}.
\end{equation*}
Using Lemma~\ref{lemma:apriori}, we see that the right-hand side  is uniformly bounded in $L^\infty(0,T;L^q(\R))$, $1\leq q\leq\infty$. It follows that
\begin{equation*}
    \alpha(k-1)\int_0^T\int_\R p_k|Q_k| \dx{x}\dx{t} \leq C(T),
\end{equation*}
as desired.

Now, using Eq.~\eqref{eq:tshiftofW} and the above computations, it is clear that $\partial_t{W_k}$ is uniformly bounded in $L^\infty(0,T;L^q(\R))$, for $1\leq q\leq\infty$. Finally we write
\begin{equation*}
    \partialt{}\partialx{W_k} = \frac{\partial^2 K}{\partial x^2}\star\prt*{p_k\partialx{W_k}} - \partialx{K}\star\prt*{p_k\frac{\partial^2 W_k}{\partial x^2}} + \frac{k-1}{\nu}\partialx{K}\star \prt*{p_k Q_k},
\end{equation*}
and use the definition of $K$, \cf Eq. \eqref{eq:SourceSolution}, to conclude the proof.
\end{proof}

\begin{remark}
All the results of this section remain valid in any spatial dimension $d\geq1$, see for example~\cite{PV15} for the a priori estimates, and the $L^1$-bound on the quantity $kp_k Q_k$.
\end{remark}

\section{Strong Compactness of the Pressure}
\label{sec:StrongCompactnessOfPressure}
This section is solely dedicated to the derivation suitable estimates in order to obtain strong compactness of the pressure, $p_k$. A key step in this pursuit  is the following BV-estimate on the individual species as well as the total population.
\begin{lemma}[Regularity of $n\i_k$ and $n_k$]
    %There exists a constant $C>0$, such that
    For $i=1,2$, we have the following
    \begin{align*}
        \abs*{\partialx {n\i_k}}, \abs*{\partialx {n_k}} \in L^\infty(0,T; L^1(\R)),
    \end{align*}
    uniformly in $k\geq2$.
\end{lemma}
\begin{proof}
For $i=1,2$, we consider
\begin{align*}
    \partialt{n_k\i} = \fpartial x \prt*{ n\i_k  \partialx {W_k}} + n\i_k G\i(p_k).
\end{align*}
Upon differentiating in space, we obtain
\begin{align}
    \label{eq:ni_x}
    \fpartial t \partialx {n\i_k} = \fpartial x \prt*{\partialx {n\i_k } \partialx W} + \fpartial x \prt*{n\i_k \frac{\partial^2 W}{\partial x ^2}} + \partialx {n\i_k} G\i(p_k) + n\i_k G_p\i(p_k) \partialx {p_k},
\end{align}
for $i=1,2$. Upon adding up both equations we get
\begin{align}
    \label{eq:n_x}
    \fpartial t \partialx {n_k} =  \fpartial x \prt*{\partialx {n_k} \partialx {W_k}} + \fpartial x\prt*{n_k \frac{\partial^2 W_k}{\partial x^2}} +  \sum_{i=1,2}  \partialx {n\i_k} G\i(p_k) + n\i_k G_p\i(p_k) \partialx {p_k}.
\end{align}
Multiplying the equation for the individual species by $\sigma\i:=\sign{(n_k\i)}$ and the equation for the total population by $\sigma:=\sign{(n_k)}$, we get, upon adding the three equations \eqref{eq:ni_x}, \eqref{eq:n_x} and integrating in space
\begin{align*}
    \ddt \int_\R &\abs*{\partialx {n\1_k}} + \abs*{\partialx {n\2_k}}+ \abs*{\partialx {n_k}} \dx{x}\\
    &= \int_\R \fpartial x \prt*{ \abs*{\partialx {n\1_k }} \partialx {W_k}} + \fpartial x \prt*{n\i_k \frac{\partial^2 W_k}{\partial x ^2}}\sigma\1 + \abs*{\partialx {n\1_k}} G\1(p_k) + \sigma\1 n\1_k G_p\1(p_k) \partialx {p_k}\\
    &\qquad + \fpartial x \prt*{ \abs*{\partialx {n\2_k }} \partialx {W_k}} + \fpartial x \prt*{n\2_k \frac{\partial^2 W_k}{\partial x ^2}}\sigma\2 + \abs*{ \partialx {n\2_k}} G\2(p_k) + \sigma\2 n\2_k G_p\2(p_k) \partialx {p_k}\\
    &\qquad +\fpartial x \prt*{ \abs*{\partialx {n_k}} \partialx {W_k}} + \fpartial x\prt*{n_k \frac{\partial^2 W_k}{\partial x^2}}\sigma  +  \sum_{i=1,2} \sigma \partialx {n\i_k} G\i(p_k) + n\i_k G_p\i(p_k) \abs*{\partialx {p_k}} \dx{x}.
\end{align*}
First we notice that the exact derivatives vanish and the estimate simplifies to 
\begin{align*}
    \ddt \int_\R &\abs*{\partialx {n\1_k}} + \abs*{\partialx {n\2_k}}+ \abs*{\partialx {n_k}} \dx{x}\\
    &\leq \int_\R \fpartial x \prt*{n\1_k \frac{\partial^2 W_k}{\partial x ^2}}\sigma\1 + \abs*{\partialx {n\1_k}} G\1(p_k) + n\1_k \abs*{G_p\1(p_k)} \abs*{\partialx {p_k}}\\
    &\qquad + \fpartial x \prt*{n\2_k \frac{\partial^2 W_k}{\partial x ^2}}\sigma\2 + \abs*{ \partialx {n\2_k}} G\2(p_k) + n\2_k \abs*{G_p\2(p_k)} \abs*{\partialx {p_k}}\\
    &\qquad  + \fpartial x\prt*{n_k \frac{\partial^2 W_k}{\partial x^2}}\sigma  +  \sum_{i=1,2} \abs*{\partialx {n\i_k}} \norm*{G\i}_\infty +  n\i_k G_p\i(p_k) \abs*{\partialx {p_k}} \dx{x}.
\end{align*}
Next, we notice that the all the terms involving the pressure gradient cancel due to opposite signs, whence
\begin{align*}
    \ddt \int_\R &\abs*{\partialx {n\1_k}} + \abs*{\partialx {n\2_k}}+ \abs*{\partialx {n_k}} \dx{x}\\
    &\leq \int_\R \fpartial x \prt*{n\1_k \frac{\partial^2 W_k}{\partial x ^2}}\sigma\1 + \abs*{\partialx {n\1_k}} G\1(p_k)\\
    &\qquad + \fpartial x \prt*{n\2_k \frac{\partial^2 W_k}{\partial x ^2}}\sigma\2 + \abs*{ \partialx {n\2_k}} G\2(p_k)\\
    &\qquad  + \fpartial x\prt*{n_k \frac{\partial^2 W_k}{\partial x^2}}\sigma  +  \sum_{i=1,2} \abs*{\partialx {n\i_k}} \norm*{G\i}_\infty  \dx{x}.
\end{align*}
Thus we are left with
\begin{align}
    \label{eq:SomeIntermediate1}
    \begin{split}
    \ddt \int_\R &\abs*{\partialx {n\1_k}} + \abs*{\partialx {n\2_k}}+ \abs*{\partialx {n_k}} \dx{x}\\
    &\leq C \int_\R\abs*{\partialx {n\1_k}} + \abs*{\partialx {n\2_k}}+ \abs*{\partialx {n_k}} \dx{x} \\
    &\quad + \int_\R \fpartial x \prt*{n\1_k \frac{\partial^2 {W_k}}{\partial x ^2}}\sigma\1  + \fpartial x \prt*{n\2_k \frac{\partial^2 {W_k}}{\partial x ^2}}\sigma\2 + \fpartial x\prt*{n_k \frac{\partial^2 W_k}{\partial x^2}}\sigma \dx{x}.
    \end{split}
\end{align}
Using the fact that
\begin{align*}
    - \nu \frac{\partial^2 W_k}{\partial x^2}  + W_k = p_k,
\end{align*}
the integrand of the last line of Eq. \eqref{eq:SomeIntermediate1} may be simplified to
\begin{align}
    \label{eq:SomeIntermediate2}
    \begin{split}
    &\fpartial x \prt*{n\1_k \frac{\partial^2 {W_k}}{\partial x ^2}}\sigma\1  + \fpartial x \prt*{n\2_k \frac{\partial^2 {W_k}}{\partial x ^2}}\sigma\2 + \fpartial x\prt*{n_k \frac{\partial^2 W_k}{\partial x^2}}\sigma\\
    &=\nu^{-1} \abs*{\partialx {n_k\1}} (W_k-p_k ) + \nu^{-1} n_k\1 \fpartial x \prt{W_k  - p_k } \sigma\1 \\
    &\quad + \nu^{-1} \abs*{\partialx {n_k\2}} (W_k-p_k ) + \nu^{-1} n_k\2 \fpartial x \prt{W_k - p_k} \sigma\2 \\
    & \quad + \nu^{-1} \abs*{\partialx {n_k}} (W_k-p_k ) + \nu^{-1} n_k \fpartial x \prt{W_k  - p_k } \sigma.
    \end{split}
\end{align}
Using the fact that $\abs{\sigma\i}, \abs{\sigma}\leq 1$ and exploiting the bounds
\begin{align*}
    n_k\i, n_k \in L^\infty(0,T; L^1(\R)), \qquad \text{and} \qquad \partialx {W_k} \in L^\infty(0,T; L^\infty(\R)),
\end{align*}
we may bound the terms of Eq. \eqref{eq:SomeIntermediate2}, and the last line of Eq. \eqref{eq:SomeIntermediate1} becomes
\begin{align}
    \label{eq:SomeIntermediate3}
    \begin{split}
    \int_\R \fpartial x &\prt*{n\1_k \frac{\partial^2 W_k}{\partial x ^2}}\sigma\1  + \fpartial x \prt*{n\2_k \frac{\partial^2 W_k}{\partial x ^2}}\sigma\2 + \fpartial x\prt*{n_k \frac{\partial^2 W_k}{\partial x^2}}\sigma\dx{x} \\
    &= C \nu^{-1} \int_R  \abs*{\partialx {n_k\1}} +   \abs*{\partialx {n_k\2}} + \abs*{\partialx {n_k}}\dx{x}\\
    &\quad + C + \nu^{-1}\int_\R  n_k\1 \abs*{\partialx  {p_k}}+ n_k\2 \abs*{\partialx {p_k}} -  n_k \abs*{\partialx  {p_k}}\dx{x}.
    \end{split}
\end{align}
The last integral  in Eq.~\eqref{eq:SomeIntermediate3} vanishes due to the fact that $n_k = n\1_k + n\2_k$. Thus, substituting Eq.~\eqref{eq:SomeIntermediate3} into Eq.~\eqref{eq:SomeIntermediate1}, an application of Gronwall's lemma yields the BV-estimate in space. 
\end{proof}

\begin{corollary}
From the proof of the preceding lemma we deduce
\begin{align}
    \int_0^T \int_\R n_k\abs*{\partialx {p_k}}\dx{x}\dx{t} \leq C,
\end{align}
where $C>0$ is independent of $k$.
\end{corollary}

\begin{proof}
Let us revisit the equation for $\partial_t {n_k}$, \ie, 
\begin{align*}
    \ddt \int_\R \abs*{\partialx {n_k}} \dx{x}
    &\leq C + C \int_\R\abs*{\partialx {n\1_k}} + \abs*{\partialx {n\2_k}}+ \abs*{\partialx {n_k}} \dx{x}\\
    &\quad + \int_\R \prt*{n\1_k G\1_p(p_k) + n\2_k G\2_p(p_k) - \nu^{-1}n_k}\abs*{\partialx {p_k}} \dx{x}.
\end{align*}
Now we use the bounds $G\i_p\leq-\alpha<0$, for $i=1,2$, and integrate in time to see that
\begin{align*}
    (\nu^{-1}+\alpha)\int_0^T\int_\R n_k\abs*{\partialx {p_k}} \dx{x}\dx{t} \leq 2\norm*{\partialx{n_k}}_{L^\infty(0,T;L^1(\R))} + C  T R,
\end{align*}
where 
\begin{align*}
    R:= \norm*{\partialx{n_k}}_{L^\infty(0,T; L^1(\R))}+\norm*{\partialx{n\1_k}}_{L^\infty(0,T; L^1(\R))}+\norm*{\partialx{n\2_k}}_{L^\infty(0,T; L^1(\R))}+1.
\end{align*}
Thus we infer that $n_k\partial_x{p_k}$ is uniformly bounded in $L^1(0,T;L^1(\R))$.
\end{proof}

\begin{lemma}[Strong Compactness of the Pressure]
    \label{lem:StrongCompactnessPressure}
    There exists a function
    \[
    p_\infty \in L^\infty\prt*{0,T; L^1(\R) \cap L^\infty(\R)},
    \] such that there holds
    \begin{align*}
        p_k \longrightarrow p_\infty,
    \end{align*}
    up to a subsequence, as $k\rightarrow \infty$ in any $L_{\mathrm{loc}}^p(0,T; L^q(\R))$, for $2\leq p, q<\infty$. In addition, the convergence also holds in the pointwise almost everywhere sense.
\end{lemma}

\begin{proof}
Let us write the quantity $n_k\abs*{\partial_x p_k}$ as a spatial derivative of a non-decreasing function of the pressure. We compute as follows
\begin{align}
    n_k \abs*{\partialx{p_k}} = \prt*{\frac{k-1}{k}}^{\frac{k}{k-1}}\sign{\prt*{\partialx{p_k}}}\partialx{}\prt*{p^{\frac{k}{k-1}}} = \prt*{\frac{k-1}{k}}^{\frac{k}{k-1}}\abs*{\partialx{}\prt*{p^{\frac{k}{k-1}}}}.
\end{align}
Let $\phi_k(z)\coloneqq z^{\frac{k}{k-1}}$. Then
\begin{align*}
    \int_0^T\int_\R n_k\abs*{\partialx {p_k}} \dx{x}\dx{t} \geq \frac14 \int_0^T\int_\R \abs*{\partialx{\phi_k}(p_k)} \dx{x}\dx{t},
\end{align*}
\ie, \ $\partial_x{\phi_k}(p_k) \in L^1(0,T;L^1(\R))$, uniformly in $k$.  Moreover, we have the same $L^1$-bound for the time derivative of $\phi_k(p_k)$. Indeed
\begin{align*}
    \partialt{\phi_k}(p_k) = \phi_k'(p_k)\partialt{p_k} = \phi_k'(p_k)\prt*{\partialx{p_k}\partialx{W_k} + \frac{k-1}{\nu}p_k Q_k} = \partialx{\phi_k}(p_k)\partialx{W_k} + \frac{k}{k-1}\frac{k-1}{\nu}p_k^{\frac1{k-1}}p_k Q_k,
\end{align*}
and therefore
\begin{align*}
    \int_0^T\int_\R \abs*{\partialt{\phi_k}(p_k)} \dx{x}\dx{t} &\leq \norm*{\partialt{\phi_k}(p_k)}_{L^1(0,T;L^1(\R))}\norm*{\partialx{W}}_{L^\infty(0,T;L^\infty(\R))} \\
    &\qquad +2p_{M}^{\frac1{k-1}}\frac{k-1}{\nu}\int_0^T\int_\R p_k \abs*{Q_k} \dx{x}\dx{t}\\
    &\leq C,
\end{align*}
where we have used Lemma~\ref{lemma:pQinL1}.

\bigskip

We conclude that the sequence $(\phi_k(p_k))_k$ converges strongly in $L^2((0,T)\times\R)$. On the other hand, from the uniform bounds on $p_k$ we infer that $p_k\rightharpoonup p_\infty$, weakly in $L^2((0,T)\times\R)$, up to the subsequence. We can therefore apply Lemma \ref{lem:CompactnessHHP} to conclude that
\begin{equation}
    \phi_k(p_k) \rightarrow p_\infty, 
\end{equation}
strongly in $L_{\mathrm{loc}}^2((0,T)\times\R)$.
We claim that this in fact implies strong convergence of the sequence of  pressures $(p_k)_k$ itself. Indeed, using the triangle inequality yields
\begin{align*}
    \norm{p_k - p_\infty}_{L^2(0,T;L^2(\R))} \leq \norm{p_k - \phi_k(p_k)}_{L^2(0,T;L^2(\R))} + \norm{\phi_k(p_k) - p_\infty}_{L^2(0,T;L^2(\R))},
\end{align*}
and
\begin{align*}
    \int_0^T\int_\R \abs*{p_k - p_k^{\frac{k}{k-1}}}^2 \dx{x}\dx{t} &= \int_0^T\int_\R \abs*{p_k}\abs*{\sqrt{p_k} - p_k^{\frac{1}{k-1}+\frac12}}^2 \dx{x}\dx{t}\\
    &\leq \sup_{0\leq z\leq p_M}\abs*{\sqrt{z} - \sqrt{z}z^{\frac1{k-1}}}^2\norm{p_k}_{L^1(0,T;L^1(\R))},
\end{align*}
with the right-hand side of the last line converging to zero. We conclude that
\begin{equation*}
    p_k \rightarrow p_\infty,
\end{equation*}
strongly in $L_{\mathrm{loc}}^2((0,T)\times\R)$. In combination with the $L^\infty$-bounds, we deduce that this convergence holds 
strongly in $L_{\mathrm{loc}}^p(0,T;L^q(\R))$, for any $2\leq p,q<\infty$, using the dominated convergence theorem. Moreover, the convergence is also true pointwise almost everywhere.
\end{proof}

\section{Incompressible Limit and Complementarity Relation}
\label{sec:Incompressible}
We have garnered all information necessary to pass to the incompressible limit  in the pressure equation \eqref{eq:PressureEqn} and prove Theorem \ref{thm:IncompressibleLimit}.

\begin{proof}[Proof of Theorem~\ref{thm:IncompressibleLimit}]
% Having established strong convergence of the sequences $(n_k)_k$ and $(p_k)_k$ we can pass to the limit in the relation
Having established strong convergence of the sequence $(p_k)_k$, and weak convergence of $(n_k)_k$ due to the a priori estimates, we can pass to the limit in the relation
\begin{equation}
    \label{eq:np}
    n_k p_k = \prt*{\frac{k-1}{k}}^{\frac{1}{k-1}}p_k^{\frac{k}{k-1}},
\end{equation}
to deduce the relation $(1-n_\infty)p_\infty = 0$, almost everywhere.
For a test function $\varphi\in C_c^1((0,T)\times\R)$, let us recall the weak formulation of the equation for the pressure
\begin{align}
    \int_0^T\int_\R\partialt \varphi p_k & - \partialx \varphi p_k  \partialx {W_k}  - \varphi p_k  \frac{\partial^2 {W_k}}{\partial x^2} \dx{x} \dx{t}\\
    &= - \int_0^T\int_\R \frac{k-1}{\nu} \varphi p_k \brk*{W_k - p_k + \nu r_k G\1(p_k) + \nu (1-r_k)G\2(p_k)} \dx{x} \dx{t}.
\end{align}
Due to the uniform bounds on the right-hand side, \cf~Lemma~\ref{lemma:pQinL1}, we may divide by $k-1$ to obtain
\begin{align*}
    0 = \lim_{k\to \infty }\int_0^T\int_\R  \varphi p_k \brk*{W_k - p_k + \nu r_k G\1(p_k) + \nu (1-r_k)G\2(p_k)} \dx{x} \dx{t}.
\end{align*}
Note that, writing $n\1_k = n_k r_k$ and expressing $n_k$ in terms of $p_k$, in a fashion similar to Eq. \eqref{eq:np}, we may readily pass to the limit in all of these terms due to the strong convergence of the pressure and the a priori bounds of Lemma~\ref{lemma:apriori}. We thus obtain
\begin{equation*}
     0 = p_\infty \prt*{W_\infty - p_\infty + \nu n\1_\infty G\1(p_\infty) + \nu n\2_\infty G\2(p_\infty)},
\end{equation*}
in the weak sense, where $n\i_\infty$ satisfies
\begin{equation*}
    \partialt{n\i_\infty} - \partialx{}{\prt*{n\i_\infty\partialx{W_\infty}}} = n\i_\infty G\i(p_\infty),
\end{equation*}
for $i=1,2$. Indeed these equations follow by passing to the limit in the weak formulation of~\eqref{eq:system}
\begin{align*}
    \int_0^T\int_\R \partialt \varphi n\i_k - \partialx \varphi n\i_k  \partialx {W_k}  \dx{x} \dx{t}
    = - \int_0^T\int_\R \varphi n\i_k G\i(p_k) \dx{x} \dx{t},
\end{align*}
where $\varphi\in C_c^1((0,T)\times\R)$.
\end{proof}

\begin{remark}
In fact, using the strategy of the previous section, \ie, the BV-bounds in space, in conjunction with a control on the time derivative obtained from bounding the right-hand side of the equation for the individual species, one can deduce also strong convergence of the sequence $(n_k)_k$. 
As a consequence, the limit functions $n_\infty, n\i_\infty$ are of bounded variation in time and space.
\end{remark}

\section{Numerical Investigations}
\label{sec:Numerics}
In this section, we revisit the results from the preceding sections and showcase certain properties of the system. The numerical simulations are performed using the positivity-preserving upwind finite volume scheme proposed for a system of two interacting species in~\cite{CHS17, CFS18} where the reaction terms are computed on each finite volume cell as simple ODEs. The implementation hinges on the fact that the elliptic Brinkman law~\eqref{eq:brinkman} can be solved using the integral representation~\eqref{eq:SourceSolution}.

\begin{figure}
    \centering
    \subfigure[Initial data, $\nu = 1$.]{
    \includegraphics[width=0.45\textwidth]{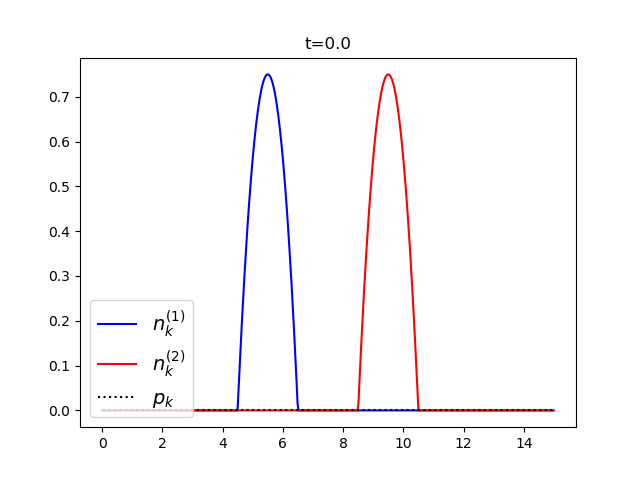}
    }
    \subfigure[Final time, $T=8$.]{
    \includegraphics[width=0.45\textwidth]{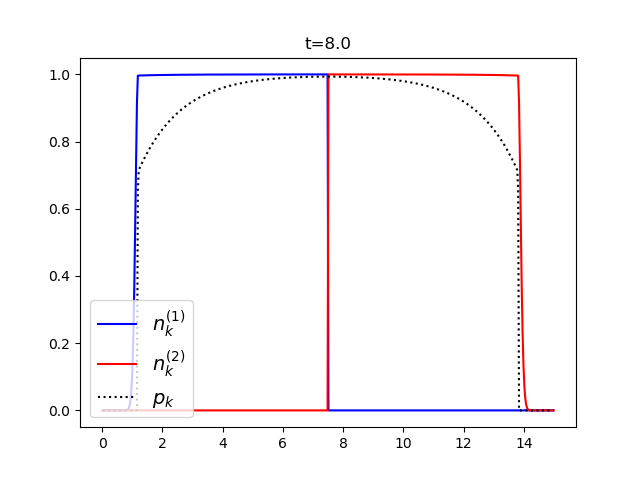}
    \label{fig:fig_1_b}
    }
    \subfigure[Initial data, $\nu=0.01$.]{
    \includegraphics[width=0.45\textwidth]{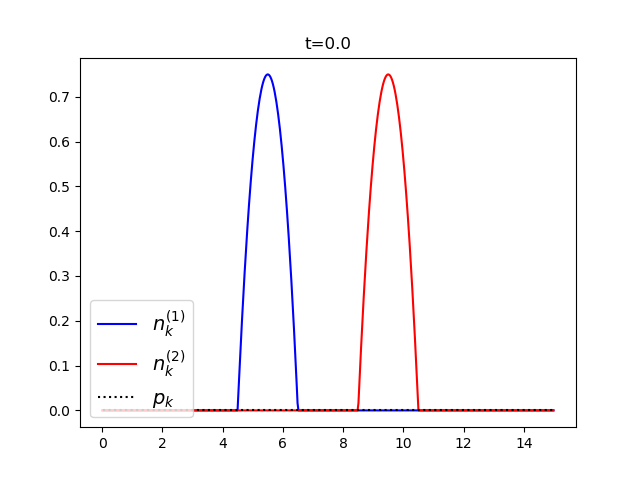}
    }
    \subfigure[Final time, $T=5$.]{
    \includegraphics[width=0.45\textwidth]{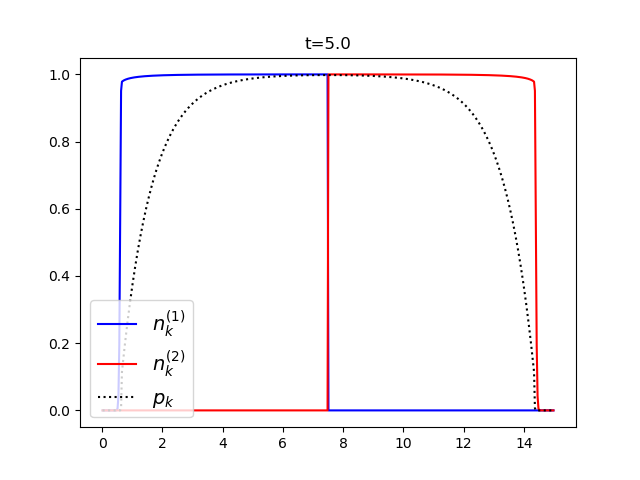}
    \label{fig:fig_1_d}
    }
    \caption{We run the simulation for the same initial data for two different values of $\nu$, \ie, $\nu = 1$ in the upper row and $\nu = 0.01$ in the bottom row. In both cases, we chose $k=100$ since we are interested in the limiting behaviour. The individual species are represented by solid lines in red and blue, the pressure is superimposed as a black dotted line. In the upper row the pressure drops to zero immediately, whereas in the bottom row we can see an almost smooth transition.}
    \label{fig:SimulationDifferentNu}
\end{figure}

Figure \ref{fig:SimulationDifferentNu}  displays the role of the viscosity parameter, $\nu$. The same initial data
\begin{align*}
    n_{k,0}\1 (x) = m (x-4.5)(6.5 - x) , \quad \text{and} \quad n_{k,0}\2 (x) = m(x-8.5)(10.5-x),
\end{align*}
are used in both cases and $m>0$ is chosen to normalise the initial mass to $1$. In both cases we used $k=100$, as we are interested in the incompressible regime. In addition, we chose $G\i(p)= 1-p$, for $i=1,2$, corresponding to a homeostatic pressure of $p_M = 1$, \cf Eq. \eqref{eq:AssumptionsOnG}. In both cases we observe the propagation of segregation in agreement with Lemma \ref{lem:IntegrabilityAndSegregation}. Moreover, we observe a drastic drop in the pressure in Figure \ref{fig:fig_1_b}. This was already observed in the one species case, \cf \cite{PV15}, where the fact was exploited that the limiting pressure has an integral representation formula. In stark contrast, Figure \ref{fig:fig_1_d} shows an almost smooth transition of the pressure indicating a much higher regularity. This is in perfect alignment with the findings of \cite{BPPS19}, as the case $\nu=0$ yields, at least formally, the system studied in the latter. As a matter of fact, the pressure gradient was shown to be  square-integrable in the Darcy case, \ie, $\nu=0$. We conclude, by remarking that the front propagation is much faster in the regime of small $\nu$, another fact that was already observed in the single-species case.

In Figure \ref{fig:NonSymmetricGrowth} we present the effect of different growth terms of the tumour cells and healthy tissue. To be more precise, we choose the same initial condition as above but use
\begin{align*}
    G\1(p) = 2 - p, \qquad \text{and} \qquad G\2(p) = 1 - p,
\end{align*}
as growth terms for the two species. We see that the first species, $n_k\1$, proliferates much faster compared to the second one. More interestingly, we see that the pressure not only has a jump at the boundaries of the support of the total population, but also at the internal layer.
\begin{figure}
    \centering
    \subfigure[Initial data, $t=0$.]{
    \includegraphics[width=0.45\textwidth]{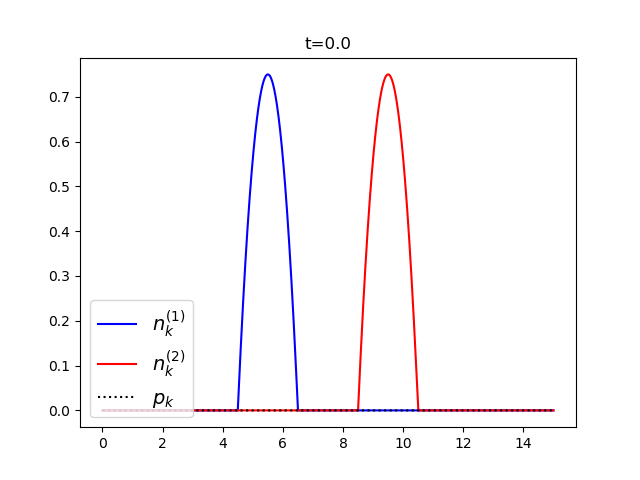}
    }
    \subfigure[Intermediate time, $t=2$.]{
    \includegraphics[width=0.45\textwidth]{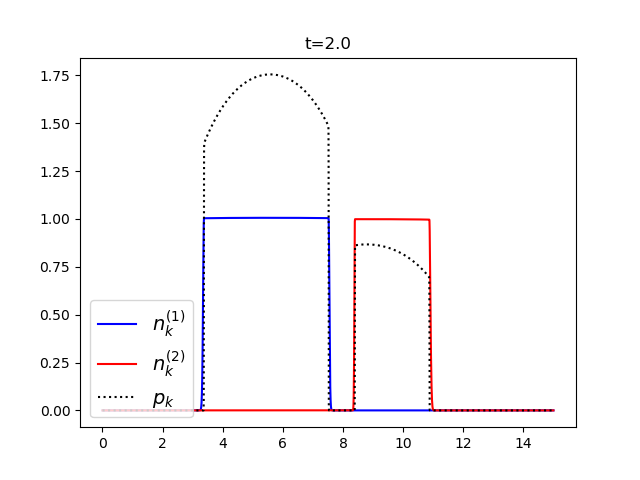}
    \label{fig:fig_2_b}
    }
    \subfigure[Intermediate time, $t=4$.]{
    \includegraphics[width=0.45\textwidth]{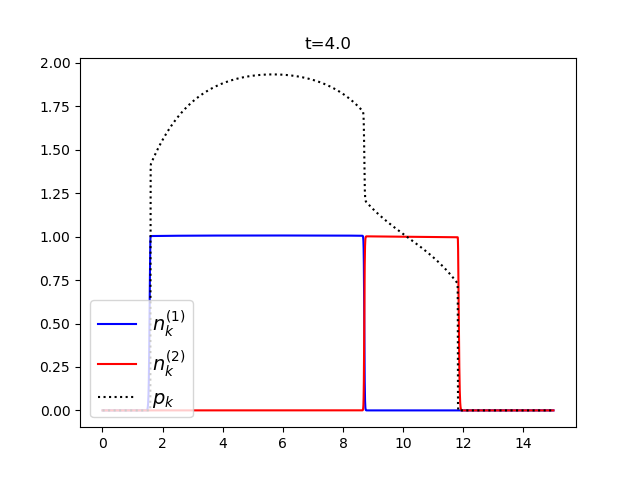}
    \label{fig:fig_2_b}
    }
    \subfigure[Final time, $T=5$.]{
    \includegraphics[width=0.45\textwidth]{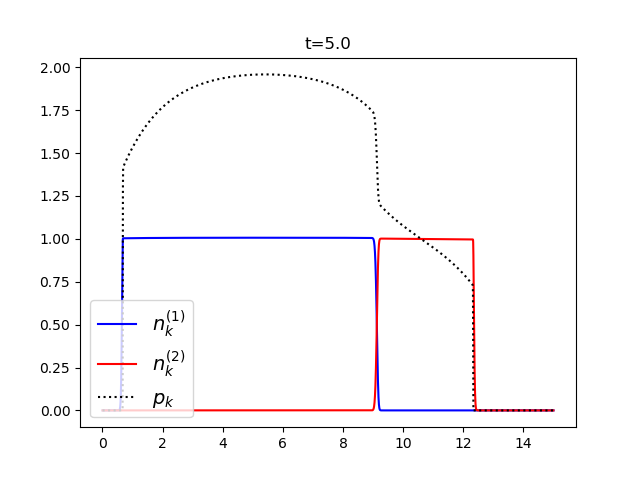}
    \label{fig:fig_2_b}
    }
    \caption{We run the simulation for the same initial data for two different growth functions, $G\i(p)$. In both cases, we chose $k=100$ and the individual species are represented by solid lines in red and blue, the pressure is superimposed as a black dotted line, as before. The pressure drops not only at the boundary of its support. We also observe jumps in internal layers.}
    \label{fig:NonSymmetricGrowth}
\end{figure}

Figure \ref{fig:Invasion} shows the evolution of system \eqref{eq:system} for initial data representing a regime where healthy tissue has already been intruded by cancerous cells, \ie,
\begin{align*}
    n_{k,0}\1 (x) = m (x-6.5)(8.5 - x) , \quad \text{and} \quad n_{k,0}\2 (x) = m(x-6)(9-x),
\end{align*}
where, again, $m>0$ normalises the mass. In addition, we choose the same unequal growth functions, $G\i$, as before, thus promoting the tumour growth compared to the normal tissue. 
\begin{figure}
    \centering
    \subfigure[Initial data, $t=0$.]{
    \includegraphics[width=0.45\textwidth]{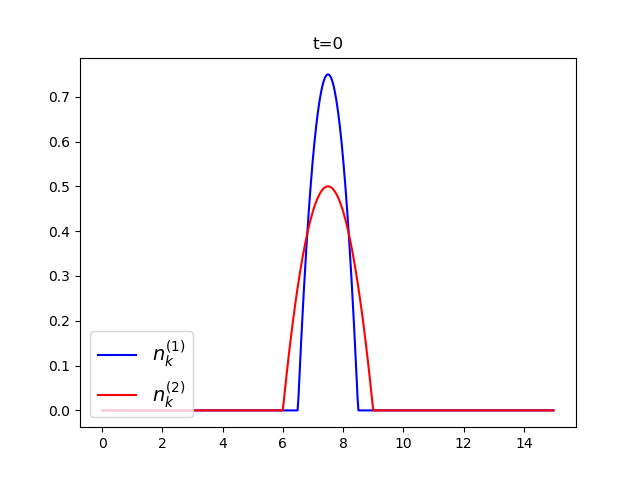}
    }
    \subfigure[Final time, $t=2$.]{
    \includegraphics[width=0.45\textwidth]{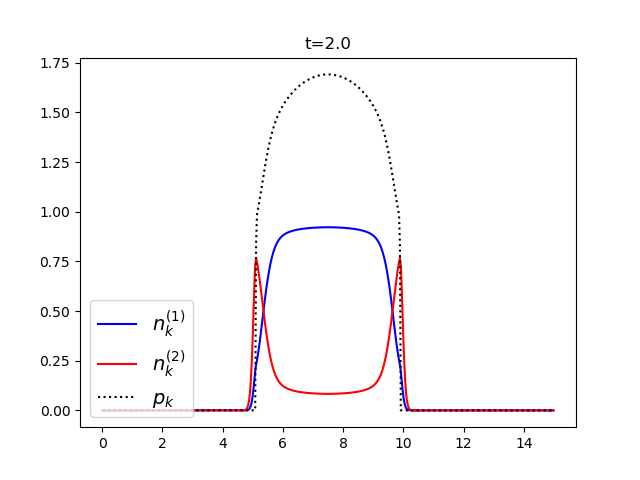}
    }
    \subfigure[Initial data, $t=4$]{
    \includegraphics[width=0.45\textwidth]{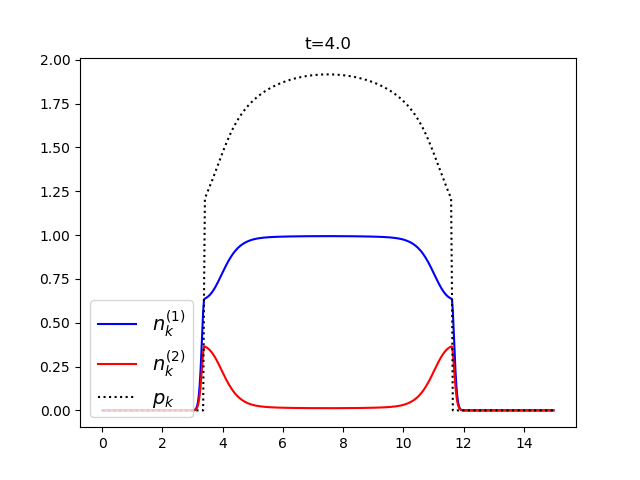}
    }
    \subfigure[Final time, $T=6$.]{
    \includegraphics[width=0.45\textwidth]{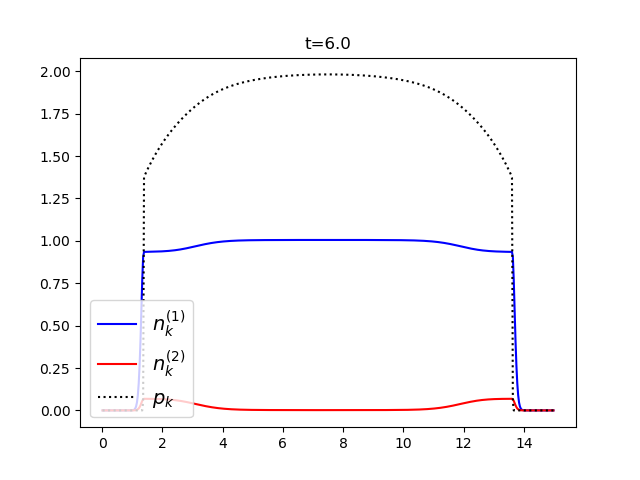}
    \label{fig:sim_5d}
    }
    \caption{The simulation shows the invasion of abnormal cells surrounded by healthy tissue. As time evolves, the tumour spreads and the density of normal cells is diminished and nearly vanishes, \cf Figure \ref{fig:sim_5d}. As before, $\nu=1$ and $k=100$.}
    \label{fig:Invasion}
\end{figure}

\section{Conclusions}
The goal of the paper was twofold. We extended an established tumour growth model to an interaction system of two cell populations, \ie, normal and abnormal cells. The interaction is given through the Brinkman flow, an elliptic equation that yields the velocity fields for each cell population. In the first part of this paper we proved the existence of solutions to the interaction system, \cf~Theorem \ref{thm:Existence}. Building upon this result, we passed to the ``incompressible'' limit in the pressure equation, Eq.~\eqref{eq:PressureEqn}, and obtained the limiting equation, also referred to as \emph{complementarity relation}, \cf~Theorem~\ref{thm:IncompressibleLimit}. This way we were able to derive a geometric model from the cell-density model we presented.

Note that both the existence result and the incompressible limit rely on strategies different from the ones adapted for related models (either in the parabolic two-species case when Brinkman's law is replaced by Darcy's law ($\nu=0$) or the one-species model with Brinkman flow). The results are complemented with a numerical investigation showcasing the segregation result, the discontinuities in the pressure and the two individual population densities which is why we do not expect better regularity than bounded variation. 

In summary, this paper extends known results in the literature to two species. As far as the existence of solutions is concerned, no additional difficulties are expected in the multi-dimensional case. However, when it comes to the stiff limit not only our method fails but also the kinetic reformulation that was employed in the one-species case, \cf \cite{PV15}, would need a serious make-over that is, at this stage, far from clear ---  even in one dimension. New singularities appear at internal layers when the two species meet and it appears different tools are required, such as the extension of the kinetic reformulation to systems, which, to our knowledge, does not exist. The exploration of such a technique is left for future works.

In addition, the rigorous inviscid limit, $\nu\rightarrow 0$, remains an open question that is left for future work.

\newpage
\section*{Appendix}
For the readers' convenience we shall recall here the compactness method invoked in \cite{HHP00} in the context of the fast reaction limit in a  cross-diffusion system with growth and death processes. Roughly speaking, it allows to  identify the limit of the composition of a uniformly compact nonlinear function and a weakly convergent sequence. 
\begin{lemma}[``Lemma A'']
\label{lem:CompactnessHHP}
Let $\Omega\subset \R$ be a compact domain and set $Q_T:=(0,T)\times\Omega$. Furthermore, let $\set{u_n}\subset L^\infty(Q_T)$ and $\set{f_n}\subset C(\R)$ be sequences with the properties
\begin{enumerate}[(i)]
    \item $u_n \rightharpoonup u$, weakly in $L^2(Q_T)$, \label{en:HHL1}
    \item $f_n$ is nondecreasing,\label{en:HHL2}
    \item $f_n \rightarrow f$, uniformly on compact subsets of $\R$, and\label{en:HHL3}
    \item $f_n(u_n) \rightarrow \chi$, strongly in $L^2(Q_T)$.\label{en:HHL4}
\end{enumerate}
Then
\begin{align*}
    \chi = f(u).
\end{align*}
\end{lemma}

For the sake of exposition, we recall here that the assumptions of the above lemma are indeed met in our case.
\begin{remark}[The assumptions are met]
The first assumption is the easiest to check as it follows directly from the uniform $L^\infty$-bounds on the pressure. Similarly, it is readily verified that each element of the sequence of functions, in our context given by $\phi_k(x) = x^{k/k-1}$, is indeed nondecreasing. Moreover, the uniform convergence towards the identity is straightforward. Thus the only requirement that needs a more minute argument is (\ref{en:HHL4}) which we present in the first part of the proof of Lemma \ref{lem:StrongCompactnessPressure}.
\end{remark}

\section*{Acknowledgements}
The authors are grateful to Beno\^it Perthame and Nicolas Vauchelet for suggesting this problem and delightful discussions.
This work was completed while T.D.\ was a visitor at Laboratoire Jacques-Louis Lions, whose kind hospitality he appreciates. 
T.D.\ recognises the support of the Polish National Agency for Academic Exchange (NAWA) and National Science Center (Poland), grant no 2018/30/M/ST1/00423. M.S. fondly acknowledges the support of the Fondation Sciences Mathématiques de Paris (FSMP) for the postdoctoral fellowship.

\end{document}